\documentclass[12pt,reqno]{amsart}

\usepackage{times}
\usepackage[T1]{fontenc}
\usepackage{mathrsfs}
\usepackage{latexsym}
\usepackage{graphicx}
\usepackage{epstopdf}
\usepackage{amsmath,amsfonts,amsthm,amssymb,amscd}
\input amssym.def
\input amssym.tex
\usepackage{color}
\usepackage{nicefrac}

\newcommand\be{\begin{equation}}
\newcommand\ee{\end{equation}}
\newcommand\bea{\begin{eqnarray}}
\newcommand\eea{\end{eqnarray}}
\newcommand\bi{\begin{itemize}}
\newcommand\ei{\end{itemize}}
\newcommand\ben{\begin{enumerate}}
\newcommand\een{\end{enumerate}}

\newtheorem{thm}{Theorem}[section]
\newtheorem{conj}[thm]{Conjecture}
\newtheorem{cor}[thm]{Corollary}
\newtheorem{lem}[thm]{Lemma}
\newtheorem{prop}[thm]{Proposition}

\newtheorem{rek}[thm]{Remark}

\newcommand{\twocase}[5]{#1 \begin{cases} #2 & \text{#3}\\ #4
&\text{#5} \end{cases}   }

\newcommand{\N}{\mathbb{N}}

\newcommand{\mattwo}[4]
{\left(\begin{array}{cc}
                        #1  & #2   \\
                        #3 &  #4
                          \end{array}\right) }

\newcommand{\vectwo}[2]
{\left(\begin{array}{c}
                        #1    \\
                        #2
                          \end{array}\right) }

\newcommand{\vectthree}[3]
{\left(\begin{array}{c}
                        #1  \\
                        #2  \\
                        #3
                          \end{array}\right) }

\numberwithin{equation}{section}

\newcounter{NLCTR}
\newenvironment{newlist}
{\begin{list}{{\bf \arabic{NLCTR}.}}
{
\usecounter{NLCTR}
\setlength{\parsep}{0.05in} \setlength{\itemsep}{\parskip}
\setlength{\leftmargin}{0.0in} \setlength{\rightmargin}{0.08in}
\setlength{\listparindent}{0in}
\setlength{\labelwidth}{-0.045in}
\setlength{\labelsep}{0.05in} \setlength{\itemindent}{0in}}}
{\end{list}}

\begin{document}

\title{Virus Dynamics on Starlike Graphs}

\author[Becker]{Thealexa Becker}
\email{tbecker@smith.edu}
\address{Department of Mathematics, Smith College, Northampton, MA 01063}

\author[Greaves-Tunnell]{Alexander Greaves-Tunnell}
\email{ahg1@williams.edu}
\address{Department of Mathematics and Statistics, Williams College, Williamstown, MA 01267}

\author[Kontorovich]{Aryeh Kontorovich}\email{lkontor@gmail.com}
\address{Department of Computer Science, Ben Gurion University of the Negev, Israel}

\author[Miller]{Steven J. Miller}
\email{sjm1@williams.edu, Steven.Miller.MC.96@aya.yale.edu} \address{Department of Mathematics and Statistics, Williams College, Williamstown, MA 01267}

\author[Ravikumar]{Pradeep Ravikumar}\email{Steven.J.Miller@williams.edu}
\address{Department of Computer Science, University of Texas Austin, Austin, TX 78701}

\author[Shen]{Karen Shen}\email{shenk@stanford.edu}
\address{Department of Mathematics, Stanford University,
Stanford, CA 94305}



\date{\today}

\thanks{The first, second and sixth named authors were partially supported by Williams College and NSF grant DMS0850577, and the fourth named author was partly supported by NSF grant DMS0970067. It is a pleasure to thank Andres Douglas Castroviejo, Amitabha Roy, and our colleagues from the Williams College 2011 SMALL REU program for many helpful conversations.}

\begin{abstract} The field of epidemiology has presented fascinating and relevant questions for mathematicians, primarily concerning the spread of viruses in a community. The importance of this research has greatly increased over time as its applications have expanded to also include studies of electronic and social networks and the spread of information and ideas. We study virus propagation on a non-linear hub and spoke graph (which models well many airline networks). We determine the long-term behavior as a function of the cure and infection rates, as well as the number of spokes $n$. For each $n$ we prove the existence of a critical threshold relating the two rates. Below this threshold, the virus always dies out; above this threshold, all non-trivial initial conditions iterate to a unique non-trivial steady state. We end with some generalizations to other networks.
\end{abstract}

\maketitle

\tableofcontents


\section{Introduction}

\subsection{Previous Work}

The general problem of studying the propagation of a node-state within a large interconnected network of nodes has a wide range of applications across domains, such as studying computer virus propagation in computer science, studying the penetration of a meme or product in marketing and sociology, and studying the propagation of an infection in epidemiology. Many of the earliest investigations \cite{Ba,KeWh,McK} assume a homogenous network, where each node has identical connections to all other nodes: for such networks, the rate of virus propagation was then shown to be determined by the density of infected nodes. While mathematically tractable, the results in \cite{FFF,RiDo, RiFoIa} also suggested that such homogenous models fail to represent many real networks. There has thus also been work on alternatives to this strict homogeneous model. For instance, \cite{P-SV1,P-SV2,P-SV3,P-SV4,MP-SV} study power law networks, where the probability of a node having $k$ neighbors is proportional to $k^{-\gamma}$ for some exponent $\gamma > 0$. Although more realistic, \cite{WKE} shows that even this model is not well-suited for many real networks. Moreover, an issue with these results is that their models, describing the propagation of node-states, themselves are dependent on the network topology. In contrast to these, \cite{WDWF} proposes a more natural topology-agnostic model that relies on local node interactions. Specifically, their proposed SIS (Susceptible Infected Susceptible) model is a discrete-time model where each node is either Susceptible (S) or Infected (I). A susceptible node is currently healthy, but at any time step can be infected by its infected neighbors. At any time step moreover, an infected node can be cured and go back to being susceptible. The model parameters are $\beta$, the probability at any time step that an infected node infects its neighbors, and $\delta$, the probability at any time step that an infected node is cured. A central set of questions given this model for propagation of a node-state through the network are:\\

\begin{enumerate}
\item Given a set of model parameters and a particular initial state, does the system then reach a steady state?\\

\item If the system does reach a steady state, what are the characteristics of that state?\\

\item What is the dynamical behavior (rate of convergence) of the system?\\
\end{enumerate}

For the SIS model, Wang et al. \cite{WDWF} gave a heuristic argument for a sufficient criterion for the node infection probabilities to converge to a trivial solution, so that the infection dies out. Using a reasonable approximation to eliminate lower order terms, they conjecture a sufficient condition for the virus to die out. For star graphs, this condition is $b \le (1-a)/\sqrt{n}$, where $a = 1 - \delta$ and $b = \beta$. One of the main contributions of this paper making this argument rigorous. Indeed, given the nonlinear coupled dynamics of the SIS model, it is typically intractable to argue rigorously about asymptotic state characteristics. But for star graphs, we are able to show that the SIS model exhibits phase transition behavior, and moreover that this threshold is both necessary and sufficient. Thus, below this threshold the virus dies out, and above the system converges to a non-trivial steady state {\em independent} of the initial state (provided only that the initial state is non-trivial). One consequence of this is that even if a single spoke node is infected initially, so long as the model parameters lie beyond the phase transition point, the infection will not die out (i.e., the node infection probabilities will not converge to the trivial point).  We prove our results through a novel two-step argument, by first reducing the problem to one with a smaller graph size, and then applying the intermediate value theorem to the dynamics over the reduced graph.


\subsection{Problem Setup}

Y. Wang, C. Deepayan, C. Wang and  C. Faloutsos \cite{WDWF} proposed the following propagation model. Denote by $\beta$, the probability at any time step that an infected node infects its neighbors, and by $\delta$, the probability at any time step that an infected node is cured.

If $p_{i, t}$ is the probability a node $i$ is infected at time $t$, the SIS model is governed by the following equation:
\be\label{eq:origmodel}
1-p_{i, t} \ = \ \left(1-p_{i, t-1}\right)\zeta_{i, t}+\delta p_{i, t}\zeta_{i, t},
\ee
where $\zeta_{i, t}$ is the probability that a node $i$ is not infected by its neighbors at time $t$. We can express $\zeta_{i, t}$ as follows:
\be\label{eq:zeta}
\zeta_{i, t} \ = \ \prod_{j \sim i}p_{j, t-1}\left(1-\beta\right)+\left(1-p_{j, t-1}\right) \ = \ \prod_{j \sim i}(1-\beta p_{j, t-1})
\ee
(where $j \sim i$ means $i$ and $j$ are neighbors --- i.e., are connected by an edge of the graph).
Given the non-linear coupled form of this system, a closed form expression for $p_{i, t}$ for the general topology case seems infeasible.

We therefore consider a specific graph topology, that of a star graph (see Figure \ref{fig:stargraph}).
\begin{figure}
\includegraphics[height=45mm]{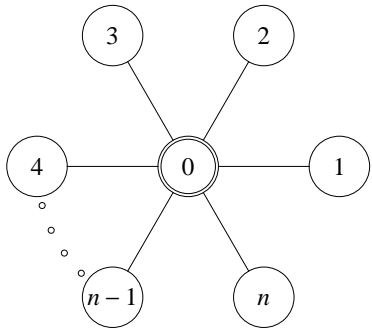}
\caption{\label{fig:stargraph} Star graph with 1 central hub and $n$ spokes.}
\end{figure}
This is a graph in which there is a single ``hub'' node which is connected to all the other nodes, the ``spokes.'' Suppose the graph has $n+1$ nodes: the hub is numbered $0$ and the spokes are numbered $1$ through $n$.

\begin{prop}
For any initial configuration, as time evolves all the spokes converge to a common behavior.
\end{prop}

\begin{proof}
\eqref{eq:origmodel} becomes
\bea
p_{0, t} & \ = \ & 1 - \left(1-p_{0, t-1}\right)\prod_{j = 1}^{n}\left(1-\beta_{j, t-1}\right)-\delta p_{0, t} \prod_{j = 1}^{n}\left(1-\beta p_{j, t-1}\right)\nonumber\\
p_{i, t} & \ = \ & 1 - \left(1-p_{i, t-1}\right)\left(1-\beta p_{0, t-1}\right) - \delta p_{i, t} \left(1-\beta p_{0, t-1}\right),
\quad 1 \ \le \  n \ \le \  n+1.
\eea

We can immediately observe that all the spokes assume identical values quite rapidly. We prove this below by showing that for $i, j \neq 0$, $\left|p_{i, t}-p_{j, t}\right| \rightarrow 0$ as $t\to\infty$. We have
\bea
p_{i, t} - p_{j, t} &\ = \ & \left(p_{i, t-1}-p_{j, t-1}\right)\left(1-\beta p_{0, t-1}\right)- \delta \left(p_{i, t}-p_{j, t}\right) \left(1-\beta p_{0, t-1}\right) \nonumber\\
& \ = \ & \left(\frac{1-\beta p_{0, t-1}}{1+\delta \left(1-\beta p_{0, t-1}\right)}\right) p_{i, t-1} - p_{j, t-1}.
\eea
Thus we have
\be
\label{eq:pijt}
\left|p_{i, t}-p_{j, t}\right|\ =\ \left(\frac{1-\beta p_{0, t-1}}{1+\delta \left(1-\beta p_{0, t-1}\right)}\right)^{t}\left|p_{i, 0}-p_{j, 0}\right|.
\ee
Since the quantity to the $t$\textsuperscript{th} power cannot stabilize at 1 as the denominator is at least $1+\delta$ and the numerator is at most 1, the right-hand side in (\ref{eq:pijt}) decays to $0$ as $t\to\infty$.
\end{proof}

An important consequence of this observation is that it allows us to simplify our model to a model in terms of $x_{t}$, the probability that the hub is infected, and $y_{t}$, the probability that a spoke is infected. These then evolve according to
\be
\vectwo{x_{t+1}}{y_{t+1}} \ = \ F\vectwo{x_{t}}{y_{t}},
\ee
where
\bea\label{eq:defnofmapF}
F(x, y) & \ = \ & \vectwo{f_{1}\left(x, y\right)}{f_{2}\left(x, y\right)} \ = \ \vectwo{1-\left(1-x\right)\left(1-\beta y\right)^{n}-\delta x \left(1-\beta y\right)^{n}}{1-\left(1-y\right)\left(1-\beta x\right) - \delta y \left(1-\beta x\right)} \nonumber\\ & \ = \ & \vectwo{1-\left(1-ax\right)\left(1-by\right)^{n}}{1-\left(1-ay\right)\left(1-bx\right)};
\eea
recall that we have defined $a := 1 - \delta$ and $b := \beta$ to simplify the algebra.

\subsection{Main Results and Consequences}

Our main result is the following.

\begin{thm}\label{thm:mainresult} Let $a,b \in (0,1)$ and $F$ as in \eqref{eq:defnofmapF} describes the limiting behavior of the spoke and star network.

\begin{newlist}

\item[I.] If $b \le (1-a)/\sqrt{n}$, then
	\begin{enumerate}
		\item[{\it (a)}] the unique fixed point of $F$ is  $(0,0)$, and
		\item[{\it (b)}] the system converges to this fixed point, that is, the virus dies out.
	\end{enumerate}
\item[II.] If $b > (1-a)/\sqrt{n}$ then, so long as the initial configuration is not the trivial point $(0,0)$,
 	\begin{enumerate}
		\item[{\it (a)}] $F$ has a unique, non-trivial fixed point $(x_f, y_f)$, where $x_f$ and $y_f$ are functions of $a,b$ and $n$, and
		\item[{\it (b)}] the system evolves to this non-trivial fixed point.
    \end{enumerate}
\end{newlist}
\end{thm}

\begin{rek} In the notation of \cite{WDWF}, the critical threshold for the epidemic is $\beta/\delta < 1/\lambda_{1,A}$, where $\lambda_{1,A}$ is the largest eigenvalue of the adjacency matrix $A$ of the network. For a star graph with $n$ spokes connected to the central hub, $\lambda_{1,A} = \sqrt{n}$. Recalling our $a=1-\delta$ and $b=\beta$, their condition is equivalent to $b = (1-a)/\sqrt{n}$, exactly the condition we have.
\end{rek}

While previous work suggested the veracity of the above claim, it was through heuristic arguments and numerical simulations. We opted for a theoretical investigation, so as to lend additional plausibility to the general conjecture and to develop some techniques potentially useful for eventually resolving it.

The proof of this theorem is distributed over the next few sections. In \S\ref{sec:detfixedpointsF}, we prove parts I(a)
and II(a) by determining the fixed points of $F$ . Using convexity arguments, we show that the trivial fixed point is the only fixed point if $b \le (1-a)/\sqrt{n}$, but there is a unique, additional fixed point for larger $b$. We prove I(b) in \S\ref{sec:convbsmall}, namely that for $b \le (1-a)/\sqrt{n}$ (so $b$ is at or below the critical threshold) all initial configurations evolve to the trivial fixed point. The proof involves linearly approximating the map $F$ near the trivial fixed point and controlling the resulting eigenvalues. Finally, we show II(b) in \S\ref{sec:convblarge}, where we prove that all non-trivial initial configurations converge to the unique non-trivial fixed point when $b > (1-a)/\sqrt{n}$.  This last case is handled by noting that there is a natural partition of the domain $[0,1]^2$ of $F$ into four regions (see Figure \ref{fig:fourregions}), where the partitions are induced from functions related to determining the location of $F$'s fixed points. The analysis of $F$ on all of $[0,1]^2$ is complicated, but the restrictions of each region lead to $F$ having simple behavior in each region. We end with a discussion of the rate of convergence and the restriction of $F$ to these regions in \S\ref{sec:behavior}, and discuss some generalizations to other graph topologies.


\section{Determination of Fixed Points of $F$}\label{sec:detfixedpointsF}

In this section we determine the behavior of the fixed points of the system as a function of the parameters $a, b$ and $n$, proving Theorem \ref{thm:mainresult}, I(a) and II(a). The proof relies on some auxiliary lemmas, which we first show. Specifically, the proofs look for partial fixed points, namely points where either the $x$ or $y$-coordinate is unchanged. We prove that the set of partial fixed points can be defined by continuous functions $\phi_{1}$ and $\phi_{2}$, whose intersections are the fixed points of the system (see Figure \ref{fig:phicurves}).




\begin{figure}
\includegraphics[height=45mm]{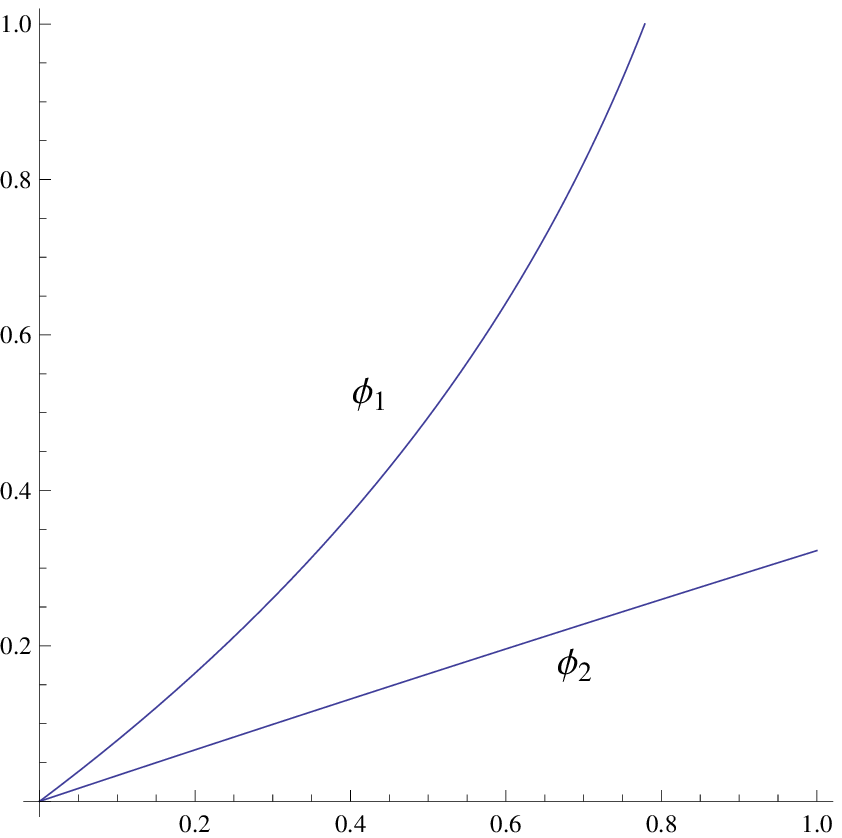}
\includegraphics[height=45mm]{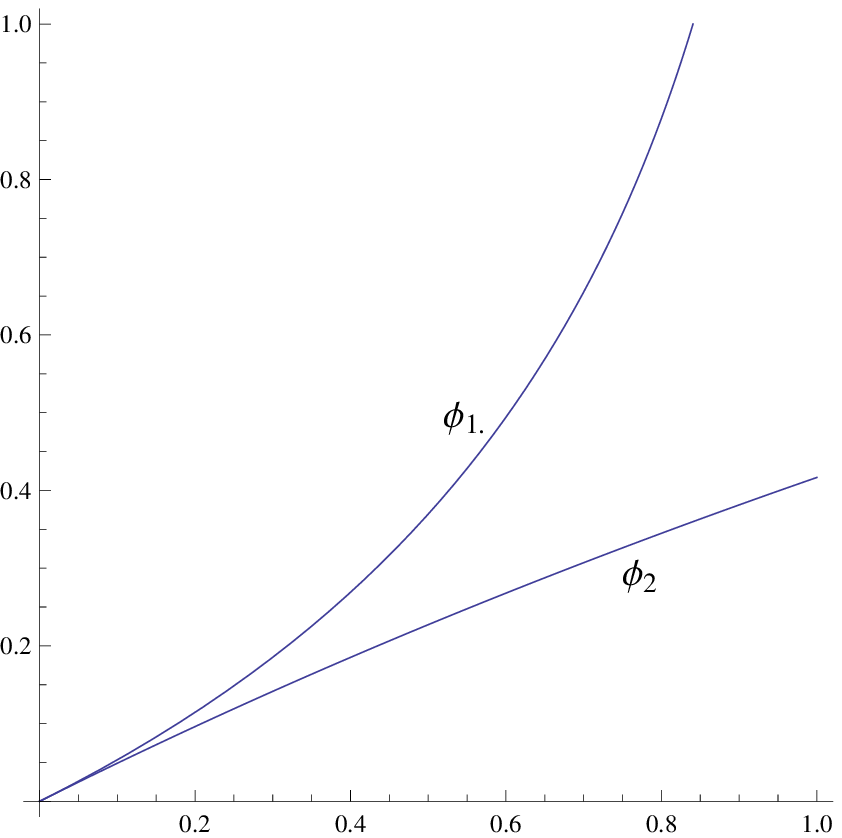}
\includegraphics[height=45mm]{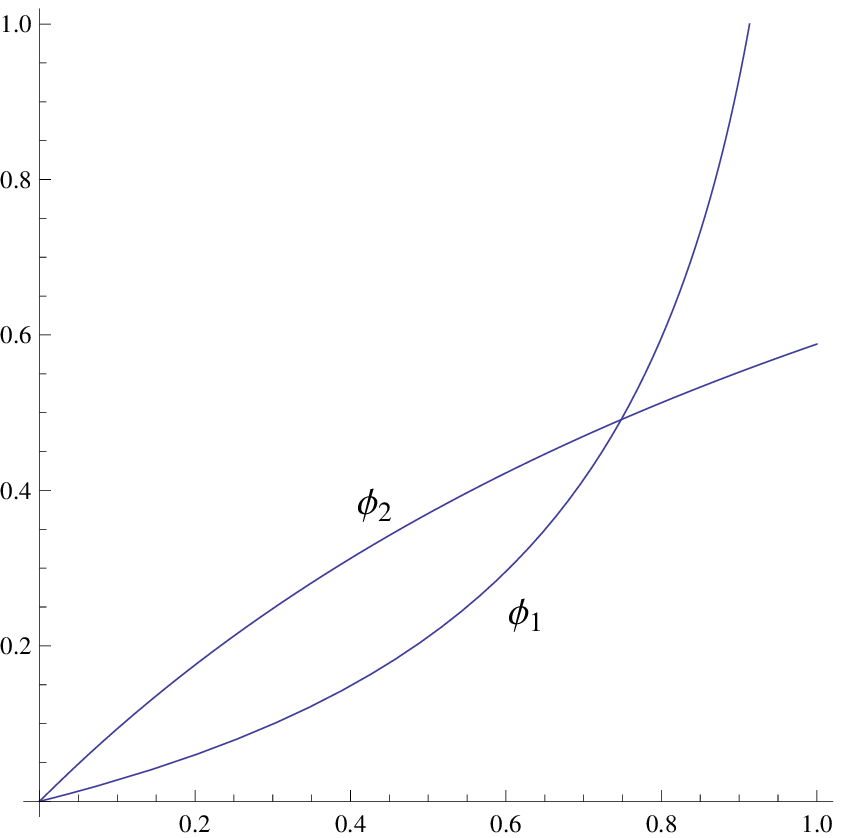}
\caption{\label{fig:phicurves} Partial fixed points from $\phi_1$ and $\phi_2$ when (from left to right) $b < (1-a)/\sqrt{n}$, $b = (1-a)/\sqrt{n}$, $b > (1-a)/\sqrt{n}$ ($b = 3, n = 4, a = .1, .4, .7$).}
\end{figure}

We begin with the following lemma characterizing these curves.

\begin{lem}\label{lem:phicurvechar} Consider the map $F$ given by \eqref{eq:defnofmapF}.

\begin{enumerate}

\item There exists a continuous, twice differentiable convex function $\phi_1: [0,1] \to [0,1]$ such that, for each $y \in [0,1]$, there is a $y' \in [0,1]$ with $F(\phi_1(y),y) = (\phi_1(y),y')$.

\item There exists a continuous, twice differentiable concave function $\phi_2: [0,1] \to [0,1]$ such that, for each $x \in [0,1]$, there is an $x' \in [0,1]$ with $F(x,\phi_2(x)) = (x',\phi_2(x))$.

\end{enumerate}
\end{lem}

\begin{proof}
We define
\be\label{eq:defng1xyfixedpoint}
g_1(x,y) \ = \ \left(1- (1-ax)(1-by)^n\right) - x
\ee
and
\be\label{eq:defng2xyfixedpoint}
g_2(x,y) \ = \ \left(1 - (1-ay)(1-bx)\right) - y.
\ee

We first analyze the set of pairs $(x,y) \in [0,1]^2$ where $g_{1}\left(x, y\right) = 0$. We immediately see that $g_1(0,0) =
0$, $g_1(0,y) > 0$ for $y \in (0,1]$, and $g_1(1,y) < 0$ for $y \in
[0,1]$. Thus by the Intermediate Value Theorem, for each $y \in
(0,1]$ there is a number (which we denote by $\phi_1(y)$) such that $g_1(\phi_1(y),y) = 0$ and
$\phi_1(y) \in [0,1]$. It is easy to see that $\phi_1(y)$ is a
continuous and differentiable function of $y$; in fact,
\bea\label{eq:defnphi1ofyexplicit} \phi_1(y) & \ = \ & \frac{1 -
(1-by)^n}{1 - a(1-by)^n} \nonumber\\ \phi_1'(y)& \ = \ &
\frac{nb(1-a)(1-by)^{n-1}}{(1 - a(1-by)^n)^2}.
\eea
Note $\phi_1(y)
\in [0,1]$: it is clearly positive, and $\frac{1-c}{1-ac} > 1$ for
$c>0$ only when $a>1$. As $a, b \in (0,1)$, $\phi_1'(y) > 0$. Thus
$\phi_1(y)$ is strictly increasing.

We analyze $g_2(x,y) = 0$ similarly. We find
\be\label{eq:defng2xyfixedpointagain} g_2(x,y) \ = \ \left(1 -
(1-ay)(1-bx)\right) - y \ = \ 0. \ee Note $g_2(0,0) = 0$, $g_2(x,0)
> 0$ for $x\in (0,1]$, and $g_2(x,1) < 0$ for $x \in [0,1]$. Solving
yields \be y \ = \ \phi_2(x) \ = \ \frac{bx}{1-a+abx}. \ee We can rewrite this as a function of $y$ as follows:
\be
x \ = \ \phi_{2}\left(y\right)\ =\ \frac{\left(1-a\right)y}{b\left(1-ay\right)}.
\ee

This is clearly continuously differentiable, and
\be \phi_2'(y) \ = \ \frac{1-a}{b(1-ay)^{2}} \ > \ 0. \ee
Thus $\phi_2(y)$ is an increasing function of $y$.

We now prove that $\phi_1\left(y\right)$ is convex and $\phi_{2}\left(y\right)$ is concave.
Straightforward differentiation and some algebra gives
\bea
\phi_1''(y)& \ = \ & -\frac{b^2n(1-a)(1-by)^{n-2}\cdot\left(n-1+a(1-by)^n +
a(n+1)(1-by)^n\right)}{(1 - a(1-by)^n)^3} \ < \ 0\nonumber\\
\phi_2''(y) & \ = \ &
\frac{2a\left(1-a\right)}{b\left(1-ay\right)^{3}} \ > \ 0.
\eea
Thus $\phi_1(y)$ is convex while $\phi_2\left(y\right)$ is concave. Direct inspection shows each function is twice continuously differentiable.
\end{proof}

The next lemma is useful in determining the number and location of fixed points of our map $F$.

\begin{lem}\label{lem:concaveconvexderivs}
Let $h_{1}, h_{2}$ be twice continuously differentiable functions such that $h_{1}\left(x\right)$ is convex and $h_{2}\left(x\right)$ is concave. If there exists some $p$ such that $h_{1}'\left(p\right) \leq h_{2}'\left(p\right)$ and $h_{1}\left(p\right) = h_{2}\left(p\right)$, then $h_{1}\left(x\right) \neq h_{2}\left(x\right)$ for all $x > p$.
\end{lem}

\begin{proof}
As $h_{1}\left(x\right)$ is convex and $h_{2}\left(x\right)$ is concave, $h_{1}'\left(x\right)$ is decreasing and $h_{2}'\left(x\right)$ is increasing. Thus, since $h_{1}'\left(p\right) \leq h_{2}'\left(p\right)$, $h_{1}'\left(x\right) < h_{2}'\left(x\right)$ for all $x > p$. As $h_{1}\left(p\right)= h_{2}\left(p\right)$, this implies that $h_{1}\left(x\right) < h_{2}\left(x\right)$ for all $x > p$.
\end{proof}

We now determine the location of the fixed points.


\begin{proof}[Proof of Theorem \ref{thm:mainresult}, I(a)]
Note that
\be\label{eq:derivsatoriginphis}
\phi_{1}'\left(0\right)\ =\ \frac{bn}{1-a}, \ \ \ \ \ \phi_{2}'\left(0\right)\ =\ \frac{1-a}{b}.
\ee
From these equations, we can see that $\phi_{2}'\left(0\right) \geq \phi_{1}'\left(0\right)$ when $b \leq (1-a)/\sqrt{n}$. Thus by Lemma \ref{lem:concaveconvexderivs}, when $b \leq (1-a)/\sqrt{n}$, there is no $y > 0$ such that $\phi_{1}\left(y\right) = \phi_{2}\left(y\right)$. The trivial fixed point is thus the unique fixed point in $[0,1]^2$.
\end{proof}

We next prove that for $b > (1-a)/\sqrt{n}$, there exists a unique non-trivial fixed point. The key ingredient is the following lemma.

\begin{lem}\label{lem:concaveconvexmax} Let $h_1, h_2 : [0,1] \to
[0,1]$ be twice continuously differentiable functions such that
$h_1(x)$ is convex, $h_2(x)$ is concave, $h_1(0) = h_2(0)=0$
and $h_1(x) \neq h_2(x)$ for $x>0$ sufficiently small. Then there exists at most one other $x>0$ for which $h_1(x) = h_2(x)$.
\end{lem}

\begin{proof} The claim is trivial if there is only one point of intersection,
so assume there are at least two. Without loss of generality we may
assume $p>0$ is the first point above zero where $h_1$ and $h_2$
agree. Such a smallest point exists by continuity, as we have
assumed $h_1(x) \neq h_2(x)$ for $x>0$ sufficiently small; if there
are infinitely many points $x_n$ where they are equal, let $p =
\liminf_n x_n > 0$.

Because $h_1(x)$ is convex, $h_1'(x)$ is increasing. By the Mean
Value Theorem there is a point $c_1 \in (0,p)$ such that
\be
h_1'(c_1)\ =\ \frac{h_1(p) - h_1(0)}{p-0}\ =\ \frac{h_{1}\left(p\right)}{p}.
\ee
As $h_1'$ is increasing, we have $h_1'(p) > h_1(c_1)$; further, $h_1'(x) > h_1(c_1)$ for all $x
\ge p$. As $h_2(x)$ is concave, $h_2'(x)$ is decreasing. Again by
the Mean Value Theorem there is a point $c_2 \in (0,p)$ such that
\be
h_2'(c_2)\ =\ \frac{h_2(p)-h_2(0)}{p-0}\ =\ \frac{h_2(p)}{p},
\ee
$h_2'(p) < h_2'(c_2)$, and $h_2'(x) < h_2'(c_2)$ for all $x \ge p$. But since $h_{1}\left(p\right) = h_{2}\left(p\right)$, $h_1'(c_1) = h_2'(c_2)$, so $h_1'(x) > h_2'(x)$ for all $x \ge p$. Thus we know from Lemma \ref{lem:concaveconvexderivs} that
there cannot be another point of intersection after $p$.
\end{proof}

We are now ready to complete the analysis.

\begin{proof}[Theorem \ref{thm:mainresult}, II(a)]
We first prove existence and then uniqueness. When $b > (1-a)/\sqrt{n}$, we know from the proof of Theorem \ref{thm:mainresult}, I(a) (see \eqref{eq:derivsatoriginphis}) that $\phi_{1}\left(y\right)$ is above $\phi_{2}\left(y\right)$ near the origin since $\phi_{1}'\left(0\right)> \phi_{2}'\left(0\right)$. The existence of the non-trivial point of intersection follows from the Intermediate Value Theorem. We recall that $y=\phi_2(x)$ is defined in $\left[0, 1\right]$ for all $x \in \left[0, 1\right]$, and $x=\phi_1(y)$ is defined in $\left[0, 1\right]$ for all $y \in \left[0, 1\right]$. As $x\to 1$ we have $\phi_2(x)$ tends to a number strictly less than 1. Thus the curve $y=\phi_2(x)$ hits the line $x=1$ below $(1,1)$. Similarly the curve $x=\phi_1(y)$ hits the line $y=1$ to the left of $(1,1)$. Thus the two curves flip as to which is above the other, implying that there must be one point where the two curves are equal.

We now have two fixed points, the trivial fixed point and the non-trivial fixed point from the second intersection of the two curves. By Lemmas \ref{lem:phicurvechar} and \ref{lem:concaveconvexmax} there are no other fixed points, and thus there is a unique, non-trivial fixed point.
\end{proof}


\section{Dynamical Behavior: $b \le (1-a)/\sqrt{n}$}\label{sec:convbsmall}

In this section we show how an eigenvalue perspective can completely determine the dynamics if $b\le (1-a)/\sqrt{n}$,
proving Theorem \ref{thm:mainresult}, I(b). As these methods fail for larger $b$, we adopt a different perspective in \S\ref{sec:convblarge}.

\subsection{Technical Preliminaries}

Our analysis of the dynamical behavior relies on the following lemma.

\begin{lem}\label{lem:techevaluelem}
Let $a, b \in (0,1)$ with $b < (1-a)/\sqrt{n}$, and let $\lambda_1
\ge \lambda_2$ denote the eigenvalues of the matrix
$\mattwo{a\alpha}{nb\beta}{b\gamma}{a\delta}$, where $\alpha, \beta,
\gamma, \delta \in [0,1]$. Then $-1 < \lambda_1, \lambda_2 < 1$.
\end{lem}

\begin{proof} The sum of the eigenvalues is the trace of the matrix (which
is $a(\alpha + \delta)$), and the product of the eigenvalues is the
determinant (which is $a^2 \alpha\delta - nb^2 \beta\gamma$). Thus
the eigenvalues satisfy the characteristic equation \be \lambda^2 -
a(\alpha+\delta)\lambda + (a^2 \alpha\delta - nb^2 \beta\gamma). \ee
The eigenvalues are therefore \be \frac{a(\alpha + \delta) \pm
\sqrt{a^2 (\alpha+\delta)^2 - 4(a^2 \alpha\delta - nb^2
\beta\gamma)}}{2} \ = \ \frac{a(\alpha + \delta) \pm \sqrt{a^2
(\alpha-\delta)^2 + 4nb^2 \beta\gamma}}{2}. \ee As the discriminant
is positive, the eigenvalues are real. Since $a(\alpha + \delta) \ge
0$, we have $|\lambda_2| \le \lambda_1$, where \be 0 \ \le \
\lambda_1 \ = \ \frac{a(\alpha + \delta) +\sqrt{a^2
(\alpha-\delta)^2 + 4nb^2 \beta\gamma}}{2}. \ee As $\beta \gamma \le
1$, $nb^2 < (1-a)^2$ and $\sqrt{u+v} \le \sqrt{u} + \sqrt{v}$ for
$u, v \ge 0$ we find \bea \lambda_1 & \ < \ & \frac{a(\alpha +
\delta) + \sqrt{a^2 (\alpha-\delta)^2} +
\sqrt{4(1-a)^2}}{2}\nonumber\\ & = & \frac{a(\alpha + \delta) +
a|\alpha-\delta| + 2(1-a)}{2} \nonumber\\ & = & \frac{2a
\max(\alpha,\delta) + 2(1-a)}{2} \nonumber\\ & = & 1 - \left(1 -
\max(\alpha,\delta)\right)a \ \le \ 1, \eea where the last claim
follows from $a, \alpha, \delta \in [0,1]$.
\end{proof}

\subsection{Proofs}

Armed with the following, we now prove the first half of our main result, the dynamical behavior at or below the critical threshold.

We prove the claim by using the Mean Value Theorem and an eigenvalue analysis of the resulting matrix. From Theorem \ref{thm:mainresult}, I(a) we know $(0,0)$ is the unique fixed point. We have \be f\left(\vectwo{u}{v}\right) \ = \
\vectwo{1-(1-au)(1-bv)^n}{1-(1-av)(1-bu)}. \ee Let \be c(t) \ = \
(1-t)\vectwo{0}{0} + t \vectwo{x}{y}, \ \ \ c'(t) \ = \
\vectwo{x}{y}. \ee Thus $c(t)$ is the line connecting the trivial
fixed point to $\vectwo{x}{y}$, with $c(0) = \vectwo{0}{0}$ and
$c(1) = \vectwo{x}{y}$. Let \be \mathcal{F}(t) \ = \ f(c(t)) \ = \
\vectwo{1-(1-atx)(1-bty)^n}{1-(1-aty)(1-btx)}. \ee Then simple
algebra (or the chain rule) yields \be \mathcal{F}'(t) \ = \
\mattwo{a(1-bty)^n}{nb(1-atx)(1-bty)^{n-1}}{b(1-aty)}{a(1-btxu)}
\vectwo{x}{y}. \ee

We now apply the one-dimensional chain rule twice, once to the
$x$-coordinate function and once to the $y$-coordinate function. We
find there are values $t_1$ and $t_2$ such that \be
f\left(\vectwo{x}{y}\right) - f\left(\vectwo{0}{0}\right) \ = \
\mattwo{a(1-bt_1y)^n}{nb(1-at_1x)(1-bt_1y)^{n-1}}{b(1-at_2y)}{a(1-bt_2x)}
\vectwo{x}{y}. \ee To see this, look at the $x$-coordinate of
$\mathcal{F}(t)$: $h(t) = 1-(1-atx)(1-bty)^n$. We have $h(1) - h(0)$
$=$ $h(1)$ $=$ $h'(t_1)(1-0)$ for some $t_1$. As \bea h'(t_1) & \ =
\ & ax(1-bt_1y)^n + nby(1-at_1x)(1-bt_1y)^{n-1} \nonumber\\
& = & \left(a(1-bt_1y)^n, \ \ nb(1-at_1x)(1-bt_1y)^{n-1}\right)
 \vectwo{x}{y}, \eea the claim follows; a similar argument
yields the claim for the $y$-coordinate (though we might have to use
a different value of $t$, and thus denote the value arising from
applying the Mean Value Theorem here by $t_2$). We therefore have \bea f\left(\vectwo{x}{y}\right) & \ =  \ &
\mattwo{a(1-bt_1y)^n}{nb(1-at_1x)(1-bt_1y)^{n-1}}{b(1-at_2y)}{a(1-bt_2x)}
\vectwo{x}{y}\nonumber\\ & \ = \ & A(a,b,x,y,t_1,t_2) \vectwo{x}{y}.
\eea

To show that $f$ is a contraction mapping, it is enough to show
that, for all $a, b$ with $b < (1-a)/\sqrt{n}$ and all $x,y \in
[0,1]$ that the eigenvalues of $A(a,b,x,y,t_1,t_2)$ are less than 1
in absolute value; however, this is exactly what Lemma
\ref{lem:techevaluelem} gives (note our assumptions imply that
$\alpha = (1-bt_1y)^n$ through $\delta = (1-bt_2x)$ are all in
$(0,1)$). Let us denote $\lambda_{\rm max}(a,b)$ the maximum value
of $\lambda_1$ for fixed $a$ and $b$ as we vary $t_1, t_2, x, y \in
[0,1]$. As we have a continuous function on a compact set, it
attains its maximum and minimum. As $\lambda_1$ is always less than
1, so is the maximum. Here it is important that we allow ourselves
to have $t_1, t_2 \in [0,1]$, so that we have a closed and bounded
set; it is immaterial (from a compactness point of view) that $a, b
\in (0,1)$ as they are fixed. As $0 < a, b < 1$,
we have $\alpha, \beta, \gamma, \delta <1$ and thus the inequalities claimed in Lemma
\ref{lem:techevaluelem} hold. For any matrix $M$ we  have $||Mv|| \le
|\lambda_{\rm max}| ||v||$; thus \be
\left|\left|f\left(\vectwo{x}{y}\right)\right|\right| \ \le \
\lambda_{\rm max}(a,b) \left|\left| \vectwo{x}{y}\right|\right|; \ee
as $\lambda_{\rm max}(a,b) < 1$ we have a contraction map. Therefore
any non-zero $\vectwo{x}{y}$ iterates to the trivial fixed point if
$b < (1-a)/\sqrt{n}$ and $n \ge 2$. In particular, the trivial fixed
point is the only fixed point (if not, $A(a,b,x,y,t_1,t_2)v = v$ for
$v$ a fixed point, but we know $||A(a,b,x,y,t_1,t_2)v|| < ||v||$ if
$v$ is not the zero vector).

\begin{rek} Unfortunately this eigenvalue approach does not work in a simple, closed form manner for general $b > (1-a)/\sqrt{n}$. We include details of such an attempted analysis in Appendix $B$.
\end{rek}


\section{Dynamical Behavior: $b > (1-a)/\sqrt{n}$}\label{sec:convblarge}

In this section we prove Theorem \ref{thm:mainresult}, II(b), establishing convergence to the non-trivial fixed point.

\subsection{Properties of the Four Regions}

Unfortunately, the method of eigenvalues does not seem to naturally generalize to large $b$. While it is possible to compute the eigenvalues of the associated matrix, it does not appear feasible to obtain a workable expression that can be understood as the parameters vary; however, breaking the analysis of $F$ into regions induced from the maps $\phi_1$ and $\phi_2$ of \S\ref{sec:detfixedpointsF} turns out to be very fruitful. This is because these curves determine partial fixed points. See Figure \ref{fig:fourregions} for the four regions.

\begin{figure}
\begin{center}
\includegraphics[totalheight=6cm]{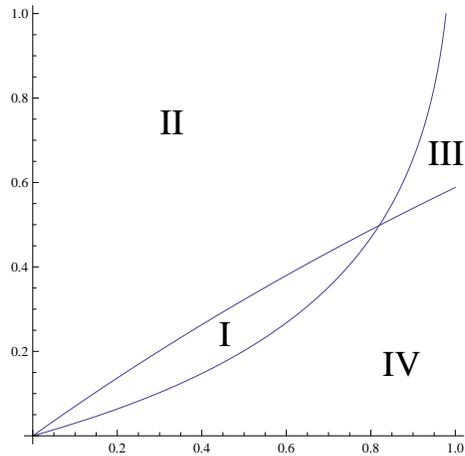}
\caption{\label{fig:fourregions} The four regions determined by $\phi_1$ and $\phi_2$ when $b>(1-a)/\sqrt{n}$.}
\end{center}\end{figure}

We first study the effect of $F$ in Regions I and III. Our first lemma provides some general information about the image of these regions under $F$, which we then use to show in the next lemma that $F$ maps each of these Regions I and III to themselves.

\begin{lem}\label{13directions} Let $b > (1-a)/\sqrt{n}$. Points in Region \emph{I} strictly increase in $x$ and $y$ on iteration by $F$, and points in Region \emph{III} strictly decrease in $x$ and $y$ on iteration.
\end{lem}

\begin{proof}
A point $(x, y)$ in Region I satisfies the inequalities
\begin{equation}\label{1x}
x\ <\ \frac{1- ( 1-by)^n}{1- a(1-by)^n}
\end{equation}
and
\begin{equation}\label{1y}
y\ <\ \frac{bx}{1-a+abx}.
\end{equation}

By multiplying by the denominator on both sides for both inequalities, we find that
\bea
x - ax\left(1-by\right)^{n} & \ <\ & 1-\left(1-by\right)^{n}
\nonumber\\ y-ay+abxy & \ <\ & bx.
\eea
Rearranging these terms gives
\begin{equation}
x\ <\ 1-(1-by)^n + ax(1-by)^n\ =\ 1-(1-ax)(1-by)^n\ =\ f_{1}\left(x,y\right)
\end{equation}
and
\begin{equation}
y\ <\ ay+bx-abxy\ =\ 1- (1-ay)(1-bx)\ =\ f_{2}\left(x,y\right).
\end{equation}
Thus, the $x$ and $y$ coordinates of the iterate of a point in Region I are strictly greater than the $x$ and $y$ coordinates of the initial point.

The proof for points in Region III is exactly analogous except with the inequalities flipped. Thus
\begin{equation}\label{3x}
x\ >\ \frac{1- ( 1-by)^n}{1- a(1-by)^n}
\end{equation}
and
\begin{equation}\label{3y}
y\ >\ \frac{bx}{1-a+abx}
\end{equation}
imply that
\begin{equation}\label{3x'}
x\ >\ 1-(1-ax)(1-by)^n\ =\ f_{1}\left(x,y\right)
\end{equation}
and
\begin{equation}\label{3y'}
y\ >\  1- (1-ay)(1-bx)\ =\ f_{2}\left(x,y\right),
\end{equation}
i.e., the $x$ and $y$ coordinates of the iterate of a point in Region III are strictly less than the $x$ and $y$ coordinates of the initial point.
\end{proof}

\begin{lem}\label{13stay} Let $b > (1-a)/\sqrt{n}$. The image of Region \emph{I} under $F$ is contained in \emph{I}, and the image of Region \emph{III} under $F$ is contained in Region \emph{III}.
\end{lem}

\begin{proof}
We prove that for a point $(x, y)$ in Region I, its iterated x-coordinate satisfies \eqref{1x} and its iterated y-coordinate satisfies \eqref{1y}.\\

\textbf{$x$-Coordinate Iteration:}

We must show that
\begin{equation}
1-(1-ax)(1-by)^n\ <\ \frac{1-(1-b(1- (1-ay)(1-bx)))^n}{1-a(1-b(1- (1-ay)(1-bx)))^n}.
\end{equation}
We'll do this by first showing the left hand side is less than $\frac{1-(1-by)^n}{1-a(1-by)^n}\ >\ 1-(1-ax)(1-by)^n$, which we then show is less than the right hand side.

Since $(x, y)$ is in Region I, we know that
\begin{equation}
x\ <\ 1-(1-ax)(1-by)^n,
\end{equation}
which implies that
\begin{equation}
\frac{x}{1-(1-ax)(1-by)^n}\ <\ 1.
\end{equation}
Since $ 0 < a, b, y < 1$, we know that $a(1-by)^n > 0$. Thus,
\begin{equation}
1- \frac{ax(1-by)^n}{1-(1-ax)(1-by)^n}\ >\ 1-a(1-by)^n.
\end{equation}
We simplify the left side of the inequality:
\begin{eqnarray}
\frac{1-(1-ax)(1-by)^n}{1-(1-ax)(1-by)^n} - \frac{ax(1-by)^n}{1-(1-ax)(1-by)^n} & \ > \ & 1-a(1-by)^n \nonumber \\ \frac{1 - (1-by)^n + ax(1-by)^n}{1-(1-ax)(1-by)^n} - \frac{ax(1-by)^n}{1-(1-ax)(1-by)^n} & > &   1-a(1-by)^n\nonumber\\  \frac{1-(1-by)^n}{1-(1-ax)(1-by)^n} & > &  1-a(1-by)^n.
\end{eqnarray}
Finally, we rearrange the inequality, and obtain our intermediate step:
\begin{equation}\label{xcoordfirsthalf}
\frac{1-(1-by)^n}{1-a(1-by)^n}\ >\ 1-(1-ax)(1-by)^n.
\end{equation}

For the second part of the proof, recall that
\begin{equation}
y\ <\ 1- (1-ay)(1-bx),
\end{equation}
which implies
\begin{equation}
(1-b(1-(1-ay)(1-bx)))^n\ <\ (1-by)^n.
\end{equation}
Now we let $(1-b(1-(1-ay)(1-bx)))^n = c$ and $(1-by)^n = c + \delta$ where $0<c<1$ and $\delta>0$ such that $c< c+\delta <1$. Then we can write
\begin{eqnarray}
-\delta &<& -a\delta \nonumber \\
1 -c-\delta-ac+ac^2+ac\delta &\ < \ &1 -c-ac+ac^2 -a\delta+a\delta c \nonumber \\
(1-ac)(1-c-\delta) &<& (1-ac-a\delta)(1-c) \nonumber \\
\frac{1-(c + \delta)}{1-a(c +\delta)} &<& \frac{1-c}{1-ac}.
\end{eqnarray}
Thus
\begin{equation}\label{xcoordsecondhalf}
\frac{1-(1-b(1- (1-ay)(1-bx)))^n}{1-a(1-b(1- (1-ay)(1-bx)))^n}\ >\ \frac{1- ( 1-by)^n}{1- a(1-by)^n}.
\end{equation}
The desired result follows from \eqref{xcoordfirsthalf} and \eqref{xcoordsecondhalf}.\\

\bigskip

\textbf{$y$-Coordinate Iteration:}

We must show that
\begin{equation}
1- (1-ay)(1-bx)\ <\ \frac{b(1-(1-ax)(1-by)^n)}{1-a+ab(1-(1-ax)(1-by)^n)}.
\end{equation} We argue similarly as before, first showing the left hand side is less than $\frac{bx}{1- (1-ay)(1-bx)}$, which we then show is less than the right hand side. Since $(x, y)$ is in Region I, we know that
\begin{equation}
y\ <\  1- (1-ay)(1-bx),
\end{equation}
which implies that
\begin{equation}
\frac{y}{1- (1-ay)(1-bx)}\ <\ 1.
\end{equation}
Since $0 < a, b, x < 1$, we know that $abx - a < 0$. Thus,
\begin{equation}
1 + \frac{y(abx - a)}{1- (1-ay)(1-bx)}\ >\ 1 - a +abx.
\end{equation}
We simplify the left side of the inequality:
\begin{eqnarray}
\frac{1- (1-ay)(1-bx)}{1- (1-ay)(1-bx)}+\frac{y(abx - a)}{1- (1-ay)(1-bx)} &>&1-a+abx \nonumber \\  \frac{ay+bx-abxy}{1- (1-ay)(1-bx)}+\frac{abxy-ay}{1- (1-ay)(1-bx)}& \ > \ & 1-a+abx  \nonumber \\ \frac{bx}{1- (1-ay)(1-bx)}&>& 1-a+abx.
\end{eqnarray}
Rearranging the inequality yields our intermediate step:
\begin{equation}\label{ycoordfirsthalf}
\frac{bx}{1 -a +abx}\ >\ 1- (1-ay)(1-bx).
\end{equation}

For the second part of the proof, recall that for a point in Region I
\begin{equation}
x\ <\ 1-(1-ax)(1-by)^n.
\end{equation}
This allows us to write $1-(1-ax)(1-by)^n = x + c$ for some $c > 0$ such that $x<x+c<1$.
Since $c > 0$ and $a,b<1$ we see that\begin{eqnarray}
bc - abc &>& 0 \nonumber \\ bx + bc - abx -abc +ab^2x^2 + ab^2xc &>& bx -abx +ab^2x^2 + ab^2xc \nonumber \\ b(x+c)(1 -a +abx) &>& bx(1 - a +ab(x+c)).
\end{eqnarray}
Thus
\begin{equation}
\frac{b(x+c)}{1-a+ab(x+c)}\ >\ \frac{bx}{1-a+abx},
\end{equation}
that is,
\begin{equation}\label{ycoordsecondhalf}
\frac{b(1-(1-ax)(1-by)^n)}{1-a+ab(1-(1-ax)(1-by)^n)}\ >\ \frac{bx}{1-a+abx}.
\end{equation}
The desired result follows from \eqref{ycoordfirsthalf} and \eqref{ycoordsecondhalf}.

The proof showing that all points in Region III iterate inside Region III under $F$ is essentially the same, now taking \eqref{3x'} and \eqref{3y'} as the initial inequalities. Thus given a point in Region III, we find that its iterated x-coordinate satisfies \eqref{3x} and its iterated y-coordinate satisfies \eqref{3y}.
\end{proof}

\subsection{Limiting Behavior}

Before proving Theorem \ref{thm:mainresult}, II(b) in general, we concentrate on the special case when the initial state is in Region I or III.

\begin{lem}\label{13fixedpoint} Let $b > (1-a)/\sqrt{n}$. All non-trivial points in Regions \emph{I} and \emph{III} iterate to the non-trivial fixed point under $F$.
\end{lem}

\begin{proof}
Consider any non-trivial point $z_{0} = (x_0, y_0)$ in Region I. Define a sequence by setting $z_{t+1} = F\left(z_{t}\right)$. By Lemma \ref{13directions}, we know that $z_{t}$ is monotonically increasing in each component, and is always in Region I. Furthermore, we know that $z_{t}$ is bounded by $(x_f, y_f)$ (the unique, non-trivial fixed point). Thus, $z_{t}$ must converge. Suppose it converges to $z'$, i.e., $\lim_{t \rightarrow \infty} z_{t} = z'$. We consider the iterate of $z'$. Since $F$ is continuous, we have
\be
F\left(z'\right)\ =\ F\left(\lim_{t\rightarrow\infty}z_{t}\right)\ =\ \lim_{t\rightarrow\infty}F\left(z_{t}\right)\ = \ \lim_{t\rightarrow\infty}z_{t+1}\ =\ \lim_{t\rightarrow\infty}z_{t}\ =\ z'.
\ee
Thus, $z'$ is a fixed point. Since $z_{0} > (0, 0)$ and $z_{t}$ is increasing, $z'$ cannot be the trivial fixed point. Thus $z'$ must be the unique non-trivial fixed point. For Region III, we have a monotonically decreasing and bounded sequence $z_{t}$ that must thus converge to a fixed point. By Lemma \ref{13stay}, this fixed point must be in Region III and thus can only be the unique non-trivial fixed point.
\end{proof}

\subsection{Proofs}

The essential idea is the following. Consider any rectangle in $[0,1]^2$ whose lower left vertex is not $(0,0)$ (the trivial fixed point introduces some complications, but we can bypass these by simply taking larger and larger rectangles). Assume the lower left and upper right vertices are in Regions I and III respectively. We show that the image of this rectangle under $F$ is strictly contained in the rectangle by showing that the image of the lower left (respectively, upper right) point has both coordinates smaller (respectively, larger) than any other iterate. As the lower left and upper right vertices iterate to the non-trivial fixed points (since they are in Regions I and III), so too do all the other points in the rectangle, as the diameters of the iterations of the rectangle tend to zero.

We make the above argument precise. Let the rectangle be all points $(x,y) \in [0,1]^2$ with $x_{\ell} \le x \le x_u$ and $y_{\ell} \le y \le y_u$. Recall $F(x,y) = (f_1(x,y), f_2(x,y))$. We choose a point $(x,y)$ in our rectangle and let $z_{0,1}(x,y) = x$ and $z_{0,2}(x,y) = y$. We define the sequence $z_t(x,y) = (z_{t,1}(x,y), z_{t,2}(x,y))$ ($t$ a positive integer) by $z_{t+1,1}(x,y)$ $=$ $f_1(z_{t,1}(x,y)$, $z_{t,2}(x,y))$ and $z_{t+1,2}(x,y)$ $=$ $f_2(z_{t,1}(x,y)$, $z_{t,2}(x,y))$. We show by induction that $z_{t,1}(x_{\ell},y_{\ell})$ $\leq$ $z_{t,1}(x,y)$ $\leq$ $z_{t,1}(x_u,y_u)$ and $z_{t,2}(x_{\ell},y_{\ell})$ $\leq$ $z_{t,2}(x,y)$ $\leq$ $z_{t,2}(x_u,y_u)$. In other words, the image of any of our rectangles is contained in the rectangle, and the lower left vertex iterates to the lower left vertex of the new region (and similarly for the top right vertex).

The base case is given by our choice of $\left(x_{\ell},y_{\ell}\right)$ and $\left(x_u,y_u\right)$, so we proceed to show the inductive step. Suppose that we have $z_{t,1}(x_{\ell},y_{\ell}) \leq z_{t,1}(x,y)$ and $z_{t,2}(x_{\ell},y_{\ell}) \le z_{t,2}(x,y)$. Then
\bea 1-az_{t,1}(x_{\ell},y_{\ell})&\ \geq\ &1-az_{t,1}(x,y) \nonumber\\
 1-bz_{t,2}(x_{\ell},y_{\ell})&\ \geq\ & z_{t,2}(x,y), \eea
which implies that
\be (1-az_{t,1}(x_{\ell},y_{\ell}))(1-bz_{t,2}(x_{\ell},y_{\ell}))^n\ \geq\ (1-az_{t,1}(x,y))(1-bz_{t,2}(x,y))^n \ee
for any $n \geq 1$. Then
\be 1-(1-az_{t,1}(x_{\ell},y_{\ell}))(1-bz_{t,2}(x_{\ell},y_{\ell}))^n\ \leq\ 1-(1-az_{t,1}(x,y))(1-bz_{t,2}(x,y))^n. \ee

That is, $z_{t+1,1}(x_{\ell},y_{\ell}) \leq z_{t+1,1}(x,y)$. Furthermore, we have that
\bea 1-az_{t,2}(x_{\ell},y_{\ell})&\ \geq\ & 1-az_{t,2}(x,y) \nonumber\\
1-bz_{t,1}(x_{\ell},y_{\ell}) &\geq &1-bz_{t,1}(x,y), \eea
which implies that
\be (1-az_{t,2}(x_{\ell},y_{\ell}))(1-bz_{t,1}(x_{\ell},y_{\ell}))\ \geq\ (1-az_{t,2}(x,y))(1-bz_{t,1}(x,y)). \ee
Then
\be 1-(1-az_{t,2}(x_{\ell},y_{\ell}))(1-bz_{t,1}(x_{\ell},y_{\ell}))\ \leq\ 1-(1-az_{t,2}(x,y))(1-bz_{t,1}(x,y)). \ee
That is, $z_{t+1,2}(x_{\ell},y_{\ell}) \leq z_{t+1,2}(x,y)$.

By a similar argument, we see that $z_{t,1}(x,y) \leq z_{t,1}(x_u,y_u)$ and $z_{t,2}(x,y) \leq z_{t,2}(x_u,y_u)$ implies that $z_{t+1,1}(x,y) \leq z_{t+1,1}(x_u,y_u)$ and $z_{t+1,2}(x,y) \leq z_{t+1,2}(x_u,y_u)$. \\

Thus $z_{t,1}(x_{\ell},y_{\ell}) \leq z_{t,1}(x,y) \leq z_{t,1}(x_u,y_u)$ and $z_{t,2}(x_{\ell},y_{\ell}) \leq z_{t,2}(x,y) \leq z_{t,2}(x_u,y_u)$ for all $t \in \N$. Taking the limit, we have
\be \lim_{t\to\infty} z_{t,1}(x_{\ell},y_{\ell}) \ \le \  \lim_{t\to\infty} z_{t,1}(x,y) \ \le \  \lim_{t\to\infty} z_{t,1}(x_u,y_u) \ee
and
\be \lim_{t\to\infty} z_{t,2}(x_{\ell},y_{\ell}) \ \le \  \lim_{t\to\infty} z_{t,2}(x,y) \ \le \  \lim_{t\to\infty} z_{t,2}(x_u,y_u) \ee

Since $\left(x_{\ell},y_{\ell}\right)$ is in Region I and $\left(x_u,y_u\right)$ is in Region III, the inequalities become
\be x_f \ \le \  \lim_{t\to\infty} z_{t,1}(x,y) \ \le \  x_f \ee
and
\be y_f \ \le \  \lim_{t\to\infty} z_{t,2}(x,y) \ \le \  y_f. \ee

Thus $\lim_{t\to\infty} z_{t,1}(x,y) = x_f$ and $\lim_{t\to\infty} z_{t,2}(x,y) = y_f$, that is,  $\left(x,y\right)$ iterates to  $\left(x_f,y_f\right)$.

We can isolate from the proof Theorem \ref{thm:mainresult}, II(b) information about the rapidity of convergence.

\begin{cor} Assume $b > (1-a)/\sqrt{n}$. Given a point $(x,y) \in (0,1)^2$, consider a rectangle with $(x,y)$ on the boundary and vertices $(x_{\rm I}, y_{\rm I})$ in Region \emph{I} and $(x_{\rm III}, y_{\rm III})$ in Region \emph{III}. Then the amount of time it takes for $(x,y)$ to converge to the unique, non-trivial fixed point is the maximum of the time it takes $(x_{\rm I}, y_{\rm I})$ and $(x_{\rm III}, y_{\rm III})$ to converge.
\end{cor}


\section{Future Research}\label{sec:behavior}

While we are able to determine the limiting behavior of any configuration, a fascinating question is to understand the path iterates take when converging to the fixed point. Based on some numerical computations and some partial theoretical results, we make the following conjecture.

\begin{conj} Let $b > (1-a)/\sqrt{n}$. Points in Regions \emph{II} and \emph{IV} exhibit one of two behaviors, depending on $a, b, n$. Either:
\begin{enumerate}
\item All points in Region \emph{II} iterate outside Region \emph{II} and all points in Region \emph{IV} iterate outside Region \emph{IV} ("flipping behavior"), or
\item All points in Region \emph{II} iterate outside Region \emph{IV} and all points in Region \emph{IV} iterate outside Region \emph{II} ("non-flipping behavior").
\end{enumerate}
\end{conj}

It would be interesting to find simple conditions involving $a, b$ and $n$ for each of the two possibilities.

Another topic for future research is to apply the methods of this paper to more general models. We present some partial results to a system which quickly follow from our arguments. We may consider star graphs with more than two levels, i.e., graphs whose spokes are themselves surrounded by additional spokes, which might themselves be surrounded by additional spokes, et cetera. We recall that \eqref{eq:origmodel} and \eqref{eq:zeta} give us the following general system:
\bea\label{eq:generalpit}
p_{i, t} & \ = \ & \left(1-p_{i, t-1}\right)\prod_{j \sim i} \left(1-\beta p_{j, t-1}\right)+ \delta p_{i, t} \prod_{j \sim i} \left(1 - \beta p_{j, t-1}\right) \nonumber\\
& \ = \ & 1 - \left(1-a p _{i, t-1}\right)\prod_{j \sim i}\left(1 - b p_{j, t-1}\right).
\eea
\emph{We keep the simplifying assumption that at each level, the number of spokes is the same.} In the $3$-level case, this means that we consider a graph with $n_{1}$ spoke nodes around a hub node, and $n_{2}$ spoke nodes around each of the $n_{1}$ spokes. Generalizing our result in the 2-dimensional case that in the limit all spokes have the same behavior, we can argue by induction that all nodes on the same `level' approach a common, limiting value. Thus, in the $\ell$-dimensional case, we are reduced to a system in $\ell$ unknowns.

We first consider the $3$-dimensional case. If we let $x_{t}$ be the probability that the hub is infected (the level 1 node), $y_{t}$ be the probability that a spoke of the hub is infected (the level 2 nodes), and $z_{t}$ be the probability a spoke of a spoke is infected (the level 3 nodes), \eqref{eq:generalpit} gives us the following system:
\be\label{eq:defof3DF}
F\vectthree{x}{y}{z} \ = \ \vectthree{1-\left(1-ax\right)\left(1-by\right)^{n_{1}}}{1-\left(1-ay\right) \left(1-bx\right)\left(1-bz\right)^{n_{2}}}{1-\left(1-az\right)\left(1-by\right)}.
\ee

We again look for partial fixed points by solving
\bea
x & \ = \ & f_{1}\left(x, y, z\right) \nonumber\\
y & \ = \ & f_{2}\left(x, y, z\right) \nonumber\\
z & \ = \ & f_{3}\left(x, y, z\right),
\eea which gives the following surfaces:
\bea
\phi_{1}\left(y, z\right) & \ = \ & x \ = \ \frac{1-\left(1-by\right)^{n_{1}}}{1-a\left(1-by\right)^{n_{1}}} \nonumber\\
\phi_{2}\left(x, z\right) & \ = \ & y \ = \ \frac{1-\left(1-bx\right)\left(1-bz\right)^{n_{2}}}{1-ay\left(1-bx\right)\left(1-bz\right)^{n_{2}}} \nonumber\\
\phi_{3}\left(x, y\right) & \ = \ & z \ = \ \frac{by}{1-a+aby}.
\eea

If we take the intersection of $\phi_{1}$ with the plane defined by $\phi_{3}$ and $\phi_{2}$ with the plane defined by $\phi_{3}$, we get two curves that look a lot like our curves from the original (2-dimensional) case. We can express these curves in terms of $x$ and $y$. The first curve is already done. For the second, we can write
\be
x \ = \ \frac{y-1}{b\left(1-ay\right)\left(1-bz\right)^{n_{2}}} + \frac{1}{b}.
\ee
Since we know that $z = by / \left(1-a+aby\right) $ we can write this as
\be
x \ = \ \frac{y-1}{b\left(1-ay\right)\left(1-\frac{b^{2}y}{1-a+aby}\right)^{n_{2}}} + \frac{1}{b}.
\ee
We now have two curves, $\phi_{1}\left(y\right)$ and $\phi_{2}\left(y\right)$. If we take their derivatives at $0$, we obtain
\bea
\phi_{1}'\left(0\right) & \ = \ & \frac{bn_{1}}{1-a} \nonumber\\
\phi_{2}'\left(0\right) & \ = \ & \frac{\left(1-a\right)^{2}-b^{2}n_{2}}{b\left(1-a\right)}.
\eea
Doing some analysis on their second derivatives shows that $\phi_{1}''\left(y\right) < 0$ and $\phi_{2}''\left(y\right) > 0$ for all $y \in \left[0, 1\right]$. Thus $\phi_{1}\left(y\right)$ is convex and $\phi_{2}\left(y\right)$ is concave. All the pieces are now in place to argue as in the proof of Theorem \ref{thm:mainresult}, I(a) and II(a). We find that there exists a unique nontrivial fixed point if and only if
\be
\phi_{1}'\left(0\right)\ >\ \phi_{2}'\left(0\right),
\ee
i.e.,
\be
b\ >\ \frac{1-a}{\sqrt{n_{1}+n_{2}}}.
\ee.

This leads to the following conjecture (which is known for $\ell = 2$ or 3).

\begin{conj} Consider a generalized spoke and star graph with $\ell$ levels. Level one consists of one node (the hub), level two consists of $n_1$ spokes connected to the central hub, and for each node of level $k$ there are $n_k$ nodes connected to it (and these are the level $k+1$ nodes). There is a unique, non-trivial fixed point if and only if $b > (1-a)/\sqrt{n_1+\cdots+n_{\ell-1}}$. \end{conj}


\appendix


\newpage

\noindent {\texttt{The following appendices highlight some of the approaches we took to tackling the problem. The first appendix describes in detail a mostly trivial analysis of the $n=1$ case, while the second appendix gives an eigenvalue approach to the problem, and the final appendix discusses some topological approaches to the problem which unfortunately did not lead to a complete solution. We include these in the arxiv version in case they may be of use to others investigating similar problems.}

\ \\


\section{Special Case: $n=1$}

The dynamical behavior can be directly determined in the special case $n=1$. Unfortunately, this is a very degenerate case, and many of the ideas and approaches here cannot be generalized to higher $n$, though some can (and in fact the analysis here was helpful in guessing some of the general behavior). In this case, it suffices to consider a one-variable problem, namely $f(x) = 1 - (1-ax)(1-bx)$. This is because when $n=1$ we cannot distinguish a spoke from the central node.

\subsection{Fixed Points}

We know from our main result that there is a unique non-trivial fixed point, but we show the proof of that result again here for the special case.

\begin{lem} The fixed points of $f$ are $0$ and
$\frac{a+b-1}{ab}$. If $a+b \le 1$ there is only one fixed point in
$[0,1]$, namely $0$. If $a+b > 1$ then there is a second fixed point
in $(0,1)$. \end{lem}

\begin{proof} We have \bea f(x) - x & \ = \ & 1-(1-ax)(1-bx) - x
\nonumber\\ & = & -abx^2 + (a+b)x - x \nonumber\\ & = & x\left(abx
- (a+b-1)\right) \nonumber\\ & = &  abx \left(x -
\frac{a+b-1}{ab}\right). \eea As the fixed points are when $f(x) -
x = 0$, the first half of the lemma is clear.

We must show $\frac{a+b-1}{ab} \in (0,1)$. Clearly we need $a+b >
1$; thus in this case $\frac{a+b-1}{ab} > 0$. To show it is at
most $1$ it suffices to show $a+b-1 < ab$ or $a+b-1 -ab < 0$. As
$a < 1$ we have \bea a + b - 1 - ab & \ = \ & a - ab + b - 1
\nonumber\\ & \ = \ & a(1-b) - (1-b) \nonumber\\ & \ = \ &
(a-1)(1-b) \ < \ 0. \eea
\end{proof}

\subsection{Derivative}

Recall $f(x) = 1 - (1-ax)(1-bx)$. Thus

\begin{lem} If $a+b \le 1$ then $|f'(x)| \le
1/2$ for all $x$; if $a+b > 1$ then $f'(x) > 0$ for all $x$.
\end{lem}

\begin{proof} We have \bea\label{eq:casen=1firstderivform} f'(x) & \ = \ &
a(1-bx) + b(1-ax) \nonumber\\ & = & (a+b)-2abx \nonumber\\ & = &
ab\left(\frac{a+b}{ab} - 2x\right). \eea Note the first derivative
is decreasing with increasing $x$.

If $a+b \le 1$ then \be |f'(x)| \ = \ |a+b - 2abx| \ < \ |1/2-(a+b)|
\ \le \ 1/2 \ee (note $a+b \le 1$ implies $ab \le 1/4$).

Assume now $a+b > 1$. When $x=0$ we have $f'(0) = a+b>1$. When $x=1$
we have $f'(1) = a+b-2ab$. Note \be a+b-2ab \ = \ a - ab + b - ab \
= \ a(1-b) + b(1-a) \ > \ 0. \ee Thus the first derivative is always
positive.
\end{proof}

\begin{rek} A trivial argument could be used to show that if $a+b \le 1$
then we have a contraction map, and everything converges to the
trivial fixed point. Thus we shall \emph{always} assume below that
$a+b > 1$, i.e., that we have a non-trivial, valid fixed point.
\end{rek}

\begin{lem} If $a+b > 1$ then we have $f'(1) < 1$. \end{lem} \begin{proof}
This follows immediately from \be\label{eq:n=1f'1} f'(1) \ = \
a(1-b) + b(1-a) \ < \ 1-b + b \ = \ 1. \ee \end{proof}

The reason it is important to note that $f'(1) < 1$ is that we
want to show that $f$ is a contraction map, at least for a subset
of $[0,1]$. Let $x_f$ denote the fixed point $\frac{a+b-1}{ab}$.
By the Mean Value Theorem we have \bea f(x) - f(x_f) \ = \
f'(\xi)(x-x_f), \ \ \xi \in [x_f,x]; \eea if $x < x_f$ then we
should write $[x,x_f]$ for the interval. As $f(x_f) = x_f$, we can
easily see what happens to a point $x$ under $f$: \be x \ \to \
f(x) \ = \ x_f + f'(\xi)(x-x_f). \ee Thus if $x$ starts
\emph{above} $x_f$ then $f(x)$ is above $x_f$ (because the
derivative is always positive and $x>x_f$); if $x$ starts
\emph{below} $x_f$ then $f(x)$ is below $x_f$ (because the
derivative is always positive and $x<x_f$).

This suggests that we should think of $f$ as a contraction map;
the problem is we need to show the existence of a $\delta \in
(0,1)$ such that $|f'(x)| \le 1 - \delta$. If this were true, then
by the Mean Value Theorem we would immediately have $f$ is a
contraction. Unfortunately, the derivative can be larger than $1$;
for example, when $x=0$ we have $f'(0) = a+b > 1$. Thus for a
small interval about $x=0$ we do not have a contraction.

We can determine where $f$ is a contraction. We must
find $x_c$ such that $f'(x_c) = 1$; as $f'$ is decreasing then the
interval $[x_c+\epsilon,1]$ will work for any $\epsilon > 0$. We
have \bea 1  \ = \  f'(x_c) \ = \  a+b -2abx_c \eea implies \be
x_c \ = \ \frac{a+b-1}{2ab} \ = \ \frac{x_f}2. \ee

\begin{lem}\label{lem:casen=1firstderivdec01} Let $a+b>1$.
The first derivative is decreasing on $[0,1]$; thus its maximum is
$f'(0) = a+b > 1$ and its minimum is $f'(1) < 1$. Further, $f'(x)
> 1$ for $x \in [0,x_c)$, $f'(x_c) = 1$ and $f'(x) < 1$ for $x \in
(x_c,1]$. Note $f'(x) > 0$.
\end{lem}

\begin{proof} That $f'(x)$ is decreasing follows from
\eqref{eq:casen=1firstderivform}; the claims on $f'(0)$ and
$f'(1)$ are immediate from the other lemmas. The rest follows from
our choice of $x_c$.
\end{proof}

\subsection{Dynamical Behavior}

Remember we define $x_c$ so that $f'(x_c) = 1$. Further $f'(x)$ is
monotonically decreasing.

\begin{thm}\label{thm:dynbehn=1case} Let $x_0 \in (0,1]$ and assume $a+b>1$.
Let $x_{m+1} = f(x_m)$. Then $\lim_{m\to\infty} x_m = x_f$, where
$x_f$ is the non-trivial, valid fixed point. \end{thm}

\begin{proof} If $x=0$ then all iterates stay at $0$. For any $\epsilon > 0$,
if $x \in [x_c + \epsilon, 1]$ then $f$ is a contraction map, and
the iterates of $x$ converge to $x_f$, the unique non-zero fixed
point. As this holds for all $\epsilon > 0$, we see that the
iterates of any $x\in (x_c,1]$ converge to $x_f$.

We are left with $x\in (0,x_c]$. As $f'(x)$ is always greater than
$1$ on $(0,x_c)$, if $x\in (0,x_c]$ then $f(x) > x$. The proof is
straightforward. By the Mean Value Theorem we have \be f(x)\ =\
f(0) + f'(\xi)x, \ \ \ \xi \in (0,x_c). \ee It is very important
that $\xi \in (0,x_c)$ and not in $[0,x_c]$. The reason is that
$f'(x) > 1$ in $(0,x_c)$ but $f'(x_c) = 1$ (see Lemma
\ref{lem:casen=1firstderivdec01}). As $f(0) = 0$ we have for all
$x\in (0,x_c]$ that \be f(x) \ = \ 0 + f'(\xi)x \ > \ x. \ee If
for some $x\in (0,x_c]$ an iterate is in $(x_c,1]$ then by earlier
arguments the future iterates converge to $x_f$.

Thus we are reduced to the case of an $x\in (0,x_c]$ such that all
iterates stay in $(0,x_c]$. We claim this cannot happen. As this is
a monotonically increasing, bounded sequence, it must converge.
Specifically, fix an $x\in (0,x_c)$. Let $x_1 = f(x)$ and in general
$x_{m+1} = f(x_m)$. Assume all $x_m \in (0,x_c)$ (if ever an $x_m =
x_c$ then $x_{m+1} = f(x_c) > x_c = x_m$ and the claim is clear).
Thus $\{x_m\}$ is a monotonically increasing bounded sequence, and
hence (compactness or the Archimedean property) converges, say to
$\widetilde{x} \le x_c$. By continuity, $\lim_{m\to\infty} x_{m+1} = \lim_{m\to\infty} f(x_m) = f(\lim_{m\to\infty} x_m)$, or $\widetilde{x} = f(\widetilde{x})$. As $\widetilde{x} > 0$, it must equal the unique, non-trivial fixed point, which cannot happen as we are assuming that all iterates are at most $x_c$. Thus some iterate exceeds $x_c$, completing the proof. \end{proof}


\begin{rek} Note the above proof required us to be very careful.
Specifically, we used the fact that $f'(x) > 1$ for $x\in (0,x_c]$
to show that such $x$ are repelled from the fixed point $0$, and
then we used the fact that $f'(x) < 1$ for $x\in (x_c,1]$ to show
such points are attracted by the non-zero fixed point $x_f$.
Arguments of this nature can be generalized.
\end{rek}



\section{Eigenvalue Approach to Fixed Points and Dynamics}\label{sec:eigenvaluefxdet}

We continue the eigenvalue approach of \S\ref{sec:convbsmall} to determining the nature of the fixed points. The following lemma will be useful.

\begin{lem}\label{lem:techlembgtoneminusamatrix}
Let $a, b \in (0,1)$, and set \be A \ = \
\mattwo{a}{nb}{b}{a}. \ee Then the eigenvalues of $A$ are $a + b
\sqrt{n}$, with corresponding eigenvector $\vectwo{\sqrt{n}}{1}$,
and $a - b \sqrt{n}$, with corresponding eigenvector
$\vectwo{-\sqrt{n}}{1}$. We may write any vector $\vectwo{x}{y}$ as
\be \vectwo{x}{y} \ = \ \left(\frac{y}2 + \frac{x}{2\sqrt{n}}\right)
\vectwo{\sqrt{n}}{1} \ + \ \left(\frac{y}2 -
\frac{x}{2\sqrt{n}}\right) \vectwo{-\sqrt{n}}{1}. \ee If $b >
(1-a)/\sqrt{n}$ then $a+b\sqrt{n} > 1$.
\end{lem}

\begin{proof} The above claims follow by direct computation. It is
convenient to write $A$ as \be A \ = \ a I + b\sqrt{n}
\mattwo{0}{\sqrt{n}}{1/\sqrt{n}}{0} \ = \ a I + b \sqrt{n} B, \ee as
the eigenvalues and eigenvectors of $B$ are easily seen by
inspection.
\end{proof}

\begin{rek}
The two eigenvectors are linearly independent, and thus a basis.
Note that any vector $v = \vectwo{x}{y}$ with positive coordinates
will have a non-zero component in the $\vectwo{\sqrt{n}}{1}$
direction. While we were able to explicitly compute the eigenvalues
and eigenvectors here, we will not need the exact values of the
eigenvectors below. From the Perron-Frobenius theorem we know that
the largest (in absolute value) eigenvalue is positive and the
corresponding eigenvector has all positive entries (because all
entries in our matrix are positive).
\end{rek}

\begin{thm} Assume $n \ge 2$, $a,b \in (0,1)$ and $b >
(1-a)/\sqrt{n}$. Then there is a $\rho = \rho(a,b,n) > 0$  such that
if $v = \vectwo{x}{y} \neq \vectwo{0}{0}$ has $||v|| \le \rho$ then
eventually an iterate of $v$ by $f$ is more than $\rho$ units form
the trivial fixed point. In other words, the trivial fixed point is
repelling.
\end{thm}

\begin{proof} We must show that if
$||v||$ is sufficiently small then there is an $m$ such that
$||f^{m}(v)|| > ||v||$, where $f^{2}(v) = f(f(v))$ and so on.

We have \bea f\left(\vectwo{u}{v}\right)& \ = \ &
\vectwo{1-(1-au)(1-bv)^n}{1-(1-av)(1-bu)} \nonumber\\ & = &
\mattwo{a}{nb}{b}{a} \vectwo{u}{v} +
O_{a,b,n}\left(\vectwo{u^2+v^2}{u^2+v^2}\right). \eea In other
words, there is some constant $C$ (depending on $n, a$ and $b$) such
that the error in replacing $f$ acting on $\vectwo{u}{v}$ by the
linear map $A = \mattwo{a}{nb}{b}{a}$ acting on $\vectwo{u}{v}$ is
at most $C\left|\left|\vectwo{u}{v}\right|\right|^2$. Thus if
$\vectwo{u}{v}$ has small length, the error will be negligible.

To show that eventually an iterate of $v=\vectwo{x}{y}$ is further
from the trivial fixed point than $v$, we argue as follows: we
replace $f$ by $A$, and since one of the eigenvalues is greater than
one eventually an iterate will be further out. The argument is
complicated by the need to do a careful book-keeping, as we must
ensure that the error terms are negligible.

Let $\lambda_1 = a + b\sqrt{n} > 1$ and $\lambda_2 = a - b\sqrt{n}$
(note $|\lambda_2| < \lambda_1$ as we have assumed $a, b > 0$). We
may write $\lambda = 1 + \eta$, with $0 < \eta < \sqrt{n}$. Our goal
is to prove an equation of the form \be f^{(m)}(v) \ = \
\lambda_1^m\left(\frac{y}2 + \frac{x}{2\sqrt{n}}\right)
\vectwo{\sqrt{n}}{1} \ + \ \lambda_2^m \left(\frac{y}2 -
\frac{x}{2\sqrt{n}}\right) \vectwo{-\sqrt{n}}{1} \ + \ {\rm small}.
\ee We often take $m$ even, so that $\lambda_2^m$ is non-negative.
We may write $x = r \cos \theta$ and $y = r \sin \theta$, with $r
\le \rho$ (later we shall determine how large $\rho$ may be).

We introduce some notation. By $E(z)$ we mean a vector
$\vectwo{z_1}{z_2}$ such that $|z_1|, |z_2| \le z$. Let $v_0 = v$
and $v_{k+1} = f(v_k)$. Thus \bea v_1  \ = \ f(v_0) \ = \  A v_0 +
E(Cr^2), \eea as $||v_0||^2 = r^2$; here $E(Cr^2)$ denotes our error
vector, which has components at most $Cr^2$. If $||v_1||
> r$ then we have found an iterate which is further from the trivial
fixed point, and we are done. If not, $||v_1|| \le r$.

Assume $||v_1|| \le r$. Then \bea v_2 \ = \ f(v_1) \ = \ Av_1 +
E(Cr^2). \eea But $Av_1 = Av_0 + A E(Cr^2)$, with $E(Cr^2)$ denoting
a vector with components at most $Cr^2$. As the largest eigenvalue
of $A$ is $\lambda_1$, we have $A E(Cr^2) = E(\lambda_1Cr^2)$. Thus
\be v_2 \ = \ A^2 v_0 + E(\lambda_1Cr^2 + Cr^2). \ee

If $||v_2|| > r$ we are done, so we assume $||v_2|| \le r$. Then
\bea v_3 \ = \ f(v_2) \ = \ Av_2 + E(Cr^2). \eea But $Av_2 = A^3 v_0
+ A E(\lambda_1Cr^2 + Cr^2)$. As \be A E(\lambda_1Cr^2 + Cr^2) \ = \
E(\lambda_1^2 Cr^2 + \lambda_1 Cr^2), \ee we find \be v_3 \ = \ A^3
v_0 + E(\lambda_1^2 Cr^2 + \lambda_1 Cr^2 + Cr^2). \ee

If there is some $m$ such that $||v_m|| > r$ then we are done. If
not, then for all $m$ we have \be v_m \ = \ A^{m} v_0 +
E\left(\sum_{k=0}^{m-1} \lambda_1^k Cr^2\right) \ = \ A^{m} v_0 +
E\left(\frac{\lambda_1^m - 1}{\lambda_1 -1}\cdot Cr^2\right). \ee
Using Lemma \ref{lem:techlembgtoneminusamatrix} (writing $v=v_0$ as
a linear combination of the eigenvectors and applying $A$) yields
\bea v_m & \ = \ & \lambda_1^m\left(\frac{y}2 +
\frac{x}{2\sqrt{n}}\right) \vectwo{\sqrt{n}}{1} \ + \ \lambda_2^m
\left(\frac{y}2 - \frac{x}{2\sqrt{n}}\right) \vectwo{-\sqrt{n}}{1}
\nonumber\\ & & \ \ \ + \ E\left(\frac{\lambda_1^m - 1}{\lambda_1
-1}\cdot Cr^2\right). \eea

We shall consider the case $x \ge y$; the other case follows
similarly. Let $m$ be the smallest even integer such that
$\lambda_1^m \ge 10$; as $\lambda_1 < 1 + \sqrt{n} < 2\sqrt{n}$ we
have for such $m$ that $\lambda_1^m \le 40n$. We consider the
$x$-coordinate of $v_m$. As $m$ is even and $x\ge y$ the
contribution from \be \lambda_1^m\left(\frac{y}2 +
\frac{x}{2\sqrt{n}}\right) \vectwo{\sqrt{n}}{1} \ + \ \lambda_2^m
\left(\frac{y}2 - \frac{x}{2\sqrt{n}}\right)
\vectwo{-\sqrt{n}}{1}\ee is at least $\lambda_1^m \cdot \frac{x
\sqrt{n}}{2\sqrt{n}} \ge 5x$; the contribution from
$E\left(\frac{\lambda_1^m - 1}{\lambda_1 -1}\cdot Cr^2\right)$ is at
most $\frac{\lambda_1^m - 1}{\lambda_1 -1}\cdot Cr^2$ $\le$
$\frac{\lambda_1^m}{\eta} \cdot Cr^2$ $\le$ $\frac{40Crn}{\eta}\cdot
r$. By assumption, $r \le \rho$. Let $\rho < \frac{\eta}{4000Cn}$.
Then the $x$-coordinate of $v_m$ is at least $4x$ (since $x \ge y$,
$x \ge r/\sqrt{2}$). Thus $||v_m||^2 \ge 16x^2 \ge 8(x^2 + y^2) =
8||v||^2 = 8r^2$, which contradicts $||v_m|| \le r$ for all $m$.

If instead $y \ge x$ then the same choices work, the only difference
being that we now look at the $y$-coordinate.
\end{proof}

Numerical exploration suggested the following conjecture (which is Theorem \ref{thm:mainresult}).

\begin{conj} Let $n=2$ and assume $a,b \in (0,1)$ with $b >
(1-a)/\sqrt{n}$. The map $f$ is a contraction map in a sufficiently
small neighborhood of the unique non-trivial valid fixed point $v_f
=\vectwo{x_f}{y_f}$. Thus, if $v = \vectwo{x}{y}$ is sufficiently
close to $v_f$, then the iterates of $v$ converge to $v_f$.
\end{conj}

While the eigenvalue approach is unable to prove the above, other techniques fared better (and in the main body of the paper we proved this by geometric arguments involving partial fixed points). Unfortunately the
linear approximation of $f$  near the non-trivial valid fixed point
$v_f$ is a horrible mess, involving numerous complicated expressions
of $a$ and $b$. While we can clean it up a bit,
it is not enough to get something which is algebraically transparent.

When $n=2$ we have \be y_f \ = \ \frac{bx_f}{1-a+abx_f}, \ \ \ x_f \
= \ \frac{(1-a)y_f}{b(1-ay_f)}. \ee Using
$f\left(\vectwo{x_f}{y_f}\right) = \vectwo{x_f}{y_f}$ yields \be
(1-bx_f) \ = \ \frac{1-y_f}{1-ay_f}, \ \ \ (1-by_f)^2 \ = \
\frac{1-x_f}{1-ax_f}. \ee These relations can help simplify some of
the formulas; the problem is the formula for $x_f$ in terms of $a$
and $b$ is a nightmare (and remember this is the `simple' case of $n=2$!): \bea x_f \ = \ \frac{2a^3 + b^3 - 2 a^2 (2 + b) + a (2 + 2 b - 2 b^2) - b\sqrt{b^4 + 4 a (1-b) (a-1-b)^2}}{2 a b
(a^2 + b^2 - a (1 + 2 b))}. \nonumber\\ \eea

The resulting fixed point matrix is \be A_f \ = \
\mattwo{a(1-by_f)^2}{2b(1-axy_f)(1-by_f)}{b(1-ay_f)}{a(1-bx_f)}. \ee
We want to show the largest eigenvalue is less than 1 in absolute
value when $b > (1-a)/\sqrt{2}$.

We know that the critical line is $b = (1-a)/\sqrt{2} = 1/\sqrt{2} -
a/\sqrt{2}$. A good way to numerically investigate the
eigenvalues of $A_f$ is study the eigenvalues along the line $b =
(m-a)/\sqrt{2}$, with $1 < m < 1 + \sqrt{2}$. This gives us a family
of parallel lines. For a given (valid) choice of $m$, we have
$\max(0, m - \sqrt{2}) < a < 1$. Below (Figures \ref{fig:plot1}
through \ref{fig:plot5}) is an illustrative set of plots of the
largest eigenvalue for 5 different choices of $m$.

\begin{figure}[ht]
\begin{center}
\scalebox{.75}{\includegraphics{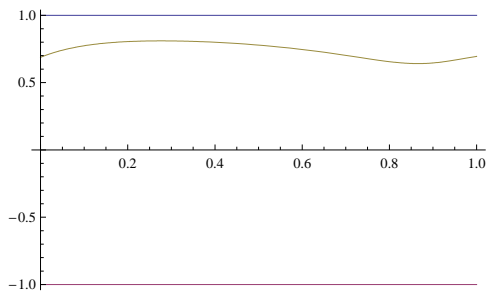}}
\caption{\label{fig:plot1} Distribution of the largest eigenvalue of
$A_f$ along the line $b = (m-a)/\sqrt{2}$, with $m = 1 + \sqrt{2}/6
\approx 1.2357$.}
\end{center}\end{figure}

\begin{figure}[ht]
\begin{center}
\scalebox{.75}{\includegraphics{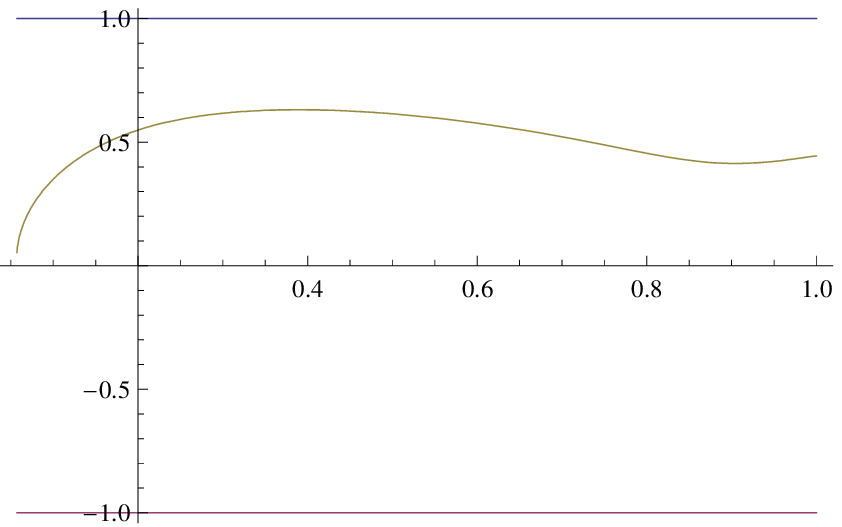}}
\caption{\label{fig:plot2} Distribution of the largest eigenvalue of
$A_f$ along the line $b = (m-a)/\sqrt{2}$, with $m = 1 + 2\sqrt{2}/6
\approx 1.4714$.}
\end{center}\end{figure}

\begin{figure}[ht]
\begin{center}
\scalebox{.75}{\includegraphics{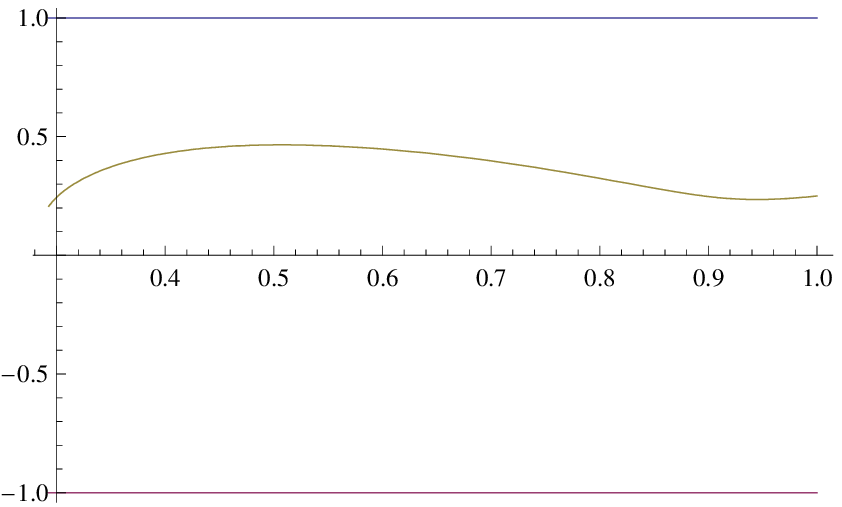}}
\caption{\label{fig:plot3} Distribution of the largest eigenvalue of
$A_f$ along the line $b = (m-a)/\sqrt{2}$, with $m = 1 + 3\sqrt{2}/6
\approx 1.7071$.}
\end{center}\end{figure}

\begin{figure}[ht]
\begin{center}
\scalebox{.75}{\includegraphics{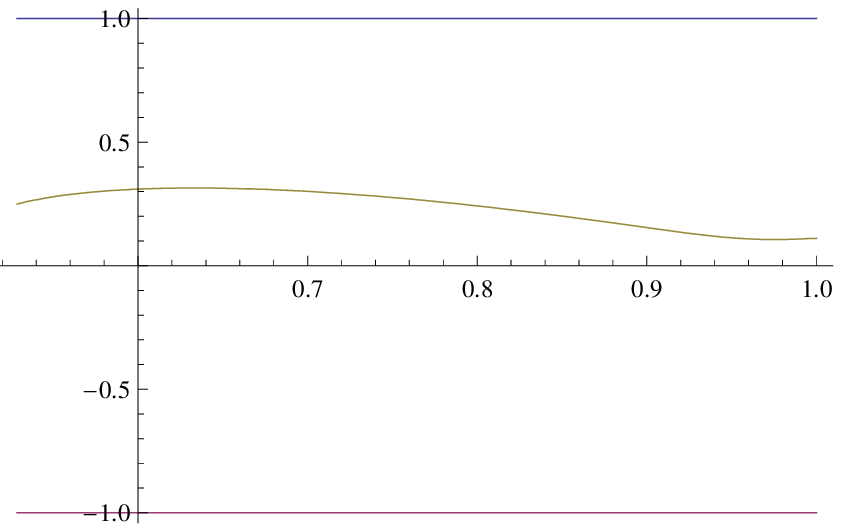}}
\caption{\label{fig:plot4} Distribution of the largest eigenvalue of
$A_f$ along the line $b = (m-a)/\sqrt{2}$, with $m = 1 + 4\sqrt{2}/6
\approx 1.9428$.}
\end{center}\end{figure}

\begin{figure}[ht]
\begin{center}
\scalebox{.75}{\includegraphics{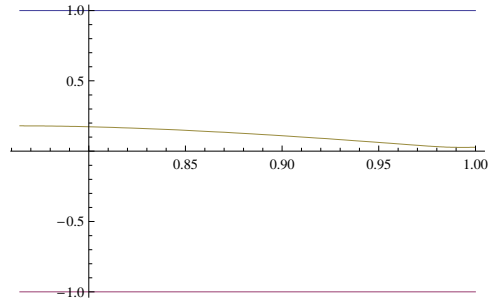}}
\caption{\label{fig:plot5} Distribution of the largest eigenvalue of
$A_f$ along the line $b = (m-a)/\sqrt{2}$, with $m = 1 + 5\sqrt{2}/6
\approx 2.1785$.}
\end{center}\end{figure}

It is crucial that $m > 1$, as $m=1$ leads to a coalescing of fixed
points (i.e., we have the trivial fixed point with multiplicity two,
and the third fixed point is not valid). In Figure \ref{fig:plot6}
we plot the behavior of $1 - \lambda_1(a,1 - \sqrt{2}/100)$, where
$\lambda_1(a,b)$ is the largest eigenvalue of $A_f$. Note that the
largest eigenvalue is very close to 1, but always less than 1, for
this value of $m$.

\begin{figure}[ht]
\begin{center}
\scalebox{.75}{\includegraphics{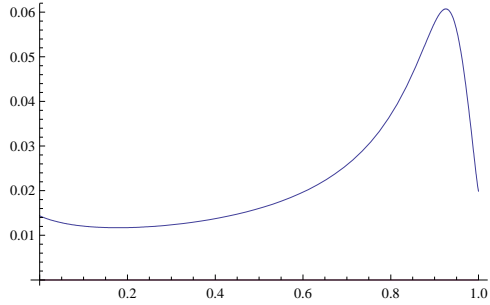}}
\caption{\label{fig:plot6} Distribution of 1 minus the largest
eigenvalue of $A_f$ along the line $b = (m-a)/\sqrt{2}$, with $m = 1
+ \sqrt{2}/100 \approx 1.0141$.}
\end{center}\end{figure}

Note in Figure \ref{fig:plot6} that $\lambda_1$ is small, especially
for large $a$. This indicates that perhaps when $a$ is close to 1
and $b = (m-a)/\sqrt{2}$ that there is a hope of proving the largest
eigenvalue is strictly less than $1$.

In fact, it is easy to show that if $a$ and $b$ are close to $1$,
then $x_f$ is close to 1 as well (which immediately implies that
$y_f$ is also close to $1$). This implies that the entries of $A_f$
are all positive numbers close to 0. A simple calculation shows \bea
\lambda_1(a,b)&\ =\ & \frac{ \left((1-by_f)^2+(1-ax_f)\right)a}{2}
\nonumber\\ & & \ \ \ +\
\frac{\sqrt{\left((1-by_f)^2-(1-ax_f)\right)a^2 +
8b^2(1-by_f)(1-ax_f)(1-ay_f)}}{2}.\nonumber\\ \eea If $a,b,x_f$ and
$y_f$ are all close to $1$, then $\lambda_1(a,b)$ will be small. We
have shown

\begin{lem} Let $n=2$, $a,b \in (0,1)$ and assume $b >
(1-a)/\sqrt{2}$. Then if $a$ and $b$ are sufficiently large, then
$f$ is a contraction map near the non-trivial valid fixed point
(i.e., the non-trivial valid fixed point is attracting). \end{lem}

With some work, using this method we can determine how `close' $a$ and $b$ need to be
to $1$. With computer assistance, we can partition $a$ and $b$ space, numerically compute the fixed points and eigenvalues, and by doing a sensitivity of parameters analysis prove the theorem.


\section{Injectivity and Topological Arguments}

One approach to this problem is to use topological arguments as a way of showing a contraction mapping and thus convergence to a unique non-trivial fixed point. Many of these arguments are facilitated by the map being injective; unfortunately, our map is only injective for some values of $a$, $b$ and $n$. In the injective cases, we can use results from topology to obtain many useful results. While these are not used in the proof of our main theorem, we include them as they may assist future researchers in studying related questions. As these cannot lead to a complete proof in general, our goal is more an exposition of these ideas then including full details.

Given that we have injectivity in certain special cases, we can analyze the dynamical behavior by using results from topology. In the special case with injectivity we can study our map on simple closed curves. This gives us the crucial property that our function maps the interior points of a simple closed curve to the interior of the image of the curve, and exterior points to the exterior. We constantly use the fact that there is a unique, non-trivial fixed point.

The presence of the trivial fixed point at $(0,0)$ causes some complications. To simplify the analysis, instead of letting $C_0$ denote the boundary of the unit square we replace the corner near $(0,0)$ with a semicircular arc from $(0,\epsilon)$ to $(\epsilon,0)$. We let $C_{n+1} = f(C_n)$, and note that $C_{n+1}$ is entirely contained in the interior of $C_n$. For $C_1$, this follows from direct computation; for $n > 2$ it follows from our injectivity assumption. As the fixed point is contained inside $C_0$, the fixed point is inside $C_n$ for all $n$ (it is the only fixed part of the interior and does not move on iteration, and thus always remains interior to every curve). This allows us to reducing the proof that all points iterate to the non-trivial fixed point to showing the sequence of boundary curves $C_n$ iterate to the fixed point.

We want to prove that the limit of $C_n$ is just a point. Unfortunately, the analysis is complicated by the fact that it is possible for the boundaries to always contract but not converge to a point. We discuss several of the potential obstructions; many of these can be eliminated by using more detailed properties of our map.

If we first assume that the image of \emph{every} curve is strictly contained in the curve, then standard arguments prove that $C_n$ converges to the non-trivial fixed point. Consider, instead, the following map. It is easier to record what happens to the radius and the angle then the point. For simplicity, we assume the non-trivial fixed point is at $(0,0)$. Given a point $(x,y)$, we write it as $(r,\theta)$. Let \be \twocase{(r,\theta)\  \longrightarrow \ }{\left(1 + \frac{r-1}{2},\theta + r\sqrt{2}\right)}{if $r \ge 1$}{\left(r + r\frac{1-r}2, \theta+r\sqrt{2}\right)}{if $r \le 1$.} \ee This is an interesting map; the origin is fixed, but all other points eventually iterate to the boundary of the unit circle. The origin is the only fixed point (the $r\sqrt{2}$ essentially gives us a rotating circle). Of course, this map violates many of the properties of our map, in particular it is not a polynomial map; however, it does have the property that all boundary curves of the region contract but do not converge to the fixed point.

The example above thus tells us that the analysis of the dynamical behavior must crucially use properties of our map, and cannot follow from general topological facts about continuous maps. Remember that the functions $\phi_1$ and $\phi_2$ divide the outer boundary of our square into four sub-regions, which are our Regions I-IV. We know that Regions I and III (save for the trivial fixed point) converge to the fixed point after successive iteration and always remain inside themselves. If any part of Region IV iterates into Regions I or III, it will converge to the fixed point, so we are not concerned with that aspect of its behavior. The difficulty is when part of Region IV iterates to Region II. Since interior and exterior cannot occupy the same space because these are all simple closed curves, all of Region I, III, and IV, would flip to outside Region IV (to see this, we separate Regions I, III and IV into one closed curve and Region II into another). Unfortunately, this leads to a complicated analysis where we start asking how many times we can have iterates of a point in IV in IV before entering II; because of these technicalities, we turned to other approaches. The interested reader can contact the authors for additional maps and examples.


\ \\

\end{document}